\documentclass[11pt]{amsart}
\usepackage[left=1.4in,right=1.4in,top=0.75in,bottom=0.75in]{geometry}

\usepackage{mathrsfs}
\usepackage{amsmath,amssymb,enumerate,tikz}
\usepackage{amsfonts,amssymb}
\usepackage{booktabs}
\usepackage[colorlinks,linkcolor=red,anchorcolor=green,
citecolor=green,CJKbookmarks=True]{hyperref}
\usepackage{xypic}%用于画交换图

\usepackage[T1]{fontenc}
\usepackage{csquotes}
\theoremstyle{plain}
\newtheorem{theorem}{Theorem}[section]
\newtheorem{lemma}[theorem]{Lemma}
\newtheorem{corollary}[theorem]{Corollary}

\newtheorem{proposition}[theorem]{Proposition}

\theoremstyle{definition}

\newtheorem{remark}[theorem]{Remark}
\newtheorem{question}{Question}

\newcommand{\MCG}{\mathcal{MCG}}
\newcommand{\ML}{\mathcal{ML}}
\newcommand{\PML}{\mathcal{PML}}

\newcommand{\MCGS}{\mathcal{MCG}^\pm(S_{g,n})}
\newcommand{\MLS}{\mathcal{ML}(S_{g,n})}
\newcommand{\PMLS}{\mathcal{PML}(S_{g,n})}

\newcommand{\C}{\mathscr{C}}
\newcommand{\T}{\mathcal{T}}
\newcommand{\R}{\mathbb R}
\newcommand{\Z}{\mathbb{Z}}

\renewcommand{\i}{\mathrm{i}}
\newcommand{\D}{\mathscr D}
\newcommand{\PD}{\mathscr{PD}}
\renewcommand{\d}{\mathrm{d}}
\newcommand{\length}{\ell}
\newcommand{\Isom}{\mathrm{Isom}}
\newcommand{\Th}{\mathrm{Th}}
\newcommand{\dth}{\mathrm{d}_{\mathrm{Th}}}
\newcommand{\TS}{\mathcal{T}(S_{g,n})}
\newcommand{\TSprime}{\mathcal{T}(S_{g',n'})}

\usepackage{soul, xcolor, todonotes, easyReview}
\title{Local Rigidity of Teichm\"uller space with the Thurston metric}
\author{Huiping Pan}
\address{School of mathematics, South China  University of Technology\\
Wushan Rd 381, Tianhe, Guangzhou, China, 510641}
\email{panhp@scut.edu.cn}
\date\today

\begin{document}

\maketitle

\begin{abstract}
  We show that every $\mathbb R$-linear surjective isometry between the {cotangent} spaces to the Teichm\"uller space equipped with the Thurston norm is induced by some isometry between the underlying hyperbolic surfaces. This is an analogue of Royden's theorem concerning the Teichm\"uller norm. 
\end{abstract}
\begin{quote}
\small{
}
\end{quote}

\section{Introduction}

\subsection{Background}

%The goal of this paper is to prove the local rigidity    of the Teichm\"uller space with the \textit{Thurston metric}, which is an analogue of the Roydens's theorem which concerns the Teichm\"uller space with the \textit{Teihm\"uller metric.}

Let $S_{g,n}$ be an oriented surface of genus $g$ with $n$ punctures whose Euler characteristic is negative: $\chi(S_{g,n})=2-2g-n<0$. Let $\TS$ and $\MCGS$  be respectively the Teichm\"uller space and  the extended mapping class group of $S_{g,n}$.

In his famous paper, Royden \cite{Royden1971} proved the \textit{infinitesimal rigidity} of the Teichm\"uller metric which states that for closed surfaces of genus at least two, every complex linear surjective isometry between the space of integrable holomorphic quadratic differentials on Riemann surfaces are induced by some biholomorphic map between the underlying Riemann surfaces. Based on this, he showed that every (global) isometry  of the Teichm\"uller space with the \textit{Teichm\"uller metric} is an element of the extended mapping class group, and that every local isometry  is the restriction of a global isometry. Since then the isometry rigidity problem of the Teichm\"uller metric has been extensively studied (see \cite{Royden1971}, \cite{EarleKra1974a}, \cite{EarleKra1974},
\cite{EarleGardiner1996}, \cite{Lakic1997}, \cite{AbatePatrizio1998}, \cite{Ivanov2001},
\cite{Markovic2003}, \cite{EarleMarkovic2003}, \cite{FarbWeinberger2010},  \cite{MineyamaMiyachi2013}). This result was extended  by Earle-Kra in  \cite{EarleKra1974a} to the case of punctured Riemann surfaces, and in \cite{EarleKra1974} to the case of isometries between different Teichm\"uller spaces. Later, it was  further extended  to bordered Riemann surfaces by Earle-Gardiner in \cite{EarleGardiner1996} and to Riemann surfaces with infinitely many punctures  by Lakic  in \cite{Lakic1997}, where in both cases the Teichm\"uller spaces are of infinite dimensional. The problem was completely settled  for general surfaces which are not of exceptional finite type by Markovic in \cite{Markovic2003}. The proofs in these papers are  based on a detailed analysis of the space of integral holomorphic quadratic differentials which serves as the cotangent space to the Teichm\"uller space.

  In \cite{Ivanov2001}, Ivanov provided a different proof of the \textit{ global isometry rigidity }  from the point of view of Mostow rigidity. A  key role in the proof is played by the \textit{curve complex}, which is an analogue of the Tits building. By analysing the asymptotic behaviour of pairs of Teichm\"uller geodesic rays, Ivanov showed that every isometry of the Teichm\"uller metric induces an automorphism of the curve complex. Then he applied the \textit{automorphism rigidity of curve complexes}  to show that every isometry is an element of the extended mapping class group.

  Ivanov's idea has been extended to other metrics on the Teichm\"uller spaces. In \cite{MasurWolf2002},  Masur-Wolf  proved the \textit{global isometry rigidity} of   the Teichm\"uller space $\T(S_{g,n})$ with the \textit{Weil-Petersson metric}, provided that  $(g,n)\notin\{(0,3),(0,4),(1,1),(1,2)\}.$ Later, it was extended by Brock-Margalit (\cite{BrockMargalit2007}) to all hyperbolic surfaces of finite area via the \textit{automorphism rigidity of pants complex} (\cite{Margalit2004}).

  In \cite{Walsh2014}, Walsh proved the \textit{global isometry rigidity} of the Teichm\"uller space $\T(S_{g,n})$ with the \textit{Thurston metric}, provided that  $(g,n)\notin\{(0,4),(1,1),(1,2)\}$.  Walsh's  proof is based on the \textit{automorphism rigidity of the curve complex} as well as the \textit{horofunction compactification} of $\T(S_{g,n})$.
  Very recently, Dumas-Lenzhen-Rafi-Tao (\cite{DLRT2016}) proved the \textit{infinitesimal rigidity}, the \textit{local isometry rigidity}, and the \textit{global isometry rigidity} of  $\T(S_{1,1})$ with the Thurston metric. The idea of their proof is to recognize lengths and intersection numbers of curves on $X$ from the features of the unit sphere in the tangent space to  $\T(S_{1,1})$, and then apply the  Fenchel-Nilsen coordinates to identify the underlying hyperbolic surfaces.

  In \cite{FarbWeinberger2010}, Farb and Weinberger  proved  a general rigidity  for all complete, finite covolume and $\MCG(S_{g,n})$ invariant Finsler metrics on $\T(S_{g,n})$ from a totally different perspective, which states that the group of isometries contains $\MCG(S_{g,n})$ as a subgroup of finite index.

The goal of this paper is to consider the \textit{infinitesimal rigidity} of the \textit{Thurston metric}.

\subsection{The Thurston metric}
The Thurston metric on $\T(S_{g,n})$ is defined as following:\footnote{Technically, Thurston defined $\dth(X,Y)$ as the logrithm of the minimum of Lipschitz constants among Lipschitz homeomorphisms from $X$ to $Y$ which are homotopic to the identity, and  proved (\ref{eq:thurston:metric}) in \cite[Theorem 8.5]{Thurston1998}. Here we use the latter as the definition of the Thurston metric for the sake of convenience.}
\begin{equation}\label{eq:thurston:metric}
  \d_{\Th}(X,Y)=\log \sup_{\alpha\in\mathscr C^{0}(S_{g,n})}
  \frac{\ell_\alpha(Y)}{\ell_\alpha(X)},
\end{equation}
where $\mathscr C^0(S_{g,n})$ is the set of (isotopy classes of) simple closed curves on $S_{g,n}$ and $\ell_\alpha(X)$ is the hyperbolic length of the geodesic representative of $\alpha$ on $X$ . Thurston showed in \cite{Thurston1998} that $\dth$ is a Finsler metric, i.e. it is induced by a (nonsymmetric) norm on the tangent bundle:
 \begin{equation}\label{eq:thurston:norm}
     \|v\|_{\mathrm{Th}}:=\sup _{\lambda \in \PML(S)}
     (d_X \log\length_\lambda)[v], \qquad \forall v\in T_X \T(S_{g,n}),
   \end{equation}
   %where $d_X \log \ell_\lambda$ is the differential of $\log \ell_\lambda$,
   where $\PML(S_{g,n})=\ML(S_{g,n})/_{\R_{>0}}$ is the space of projective measured laminations and $\ML(S_{g,n})$ is the space of measured laminations on $S_{g,n}$. We call the above asymmetric norm the \textit{Thurston infinitesimal norm}.

   \begin{theorem}[\cite{DLRT2016}, Theorem 6.1]
   \label{thm:regularity}
      Let $S_{g,n}$ be an orientable surface of finite  type with negative Euler characteristic. Then the Thurston norm function $T \TS\to\R$ is locally Lipschitz.
   \end{theorem}

    The unit sphere in the cotangent bundle has a very nice description as following.
 \begin{theorem}[\cite{Thurston1998}, Theorem 5.1]\label{thm:convex:model}
   For any hyperbolic surface $X$ of finite type, the map
   \begin{equation*}
     \begin{array}{cccc}
       \PD_X: & \PMLS & \longrightarrow & T^*_X\TS\\
        & [\mu] & \longmapsto & d_X \log \ell_{\mu}
     \end{array}
   \end{equation*}
   embeds $\PMLS$ as the boundary of a convex neighbourhood of the origin.  This convex neighbourhood is dual to the unit ball $\{v\in T_X\TS:\|v\|_{\mathrm{Th}}\leq1\}$.
 \end{theorem}

   The following rigidity questions were raised by Papadopoulos, Ther\'et, Dumas, and Rafi  in
  \cite[Problem IV]{PapadopoulosTheret2007} and  \cite[Problem 2.1, Problem 2.6]{Su2016}.
 \begin{question}\label{question1}
 \begin{enumerate}
   \item Does the Thurston infinitesimal norm  determine the underlying hyperbolic surfaces?
    \item Is each local isometry of Teichm\"uller space with the Thurston metric induced by an element of the extended mapping class group?
 \end{enumerate}
 \end{question}
 \begin{remark}
 For $\T(S_{1,1})$,
 Question \ref{question1} was answered affirmatively  by Dumas-Lenzhen-Rafi-Tao  in \cite{DLRT2016}.
 \end{remark}

 %By identifying $\PMLS$ with $\{\mu\in\MLS:\ell_\mu(X)=1\}$, we obtain a homeomorphism
  %\begin{equation*}
  %   \begin{array}{cccc}
  %     \D_X: & \MLS & \longrightarrow & T^*_X\TS-\{0\}\\
  %      & \mu & \longmapsto & d_X \ell_\mu.
  %   \end{array}
  % \end{equation*}
 %  In particular, $\|\D_X(\mu)\|_{\mathrm{Th}}=\ell_\mu(X)$.
%   In this way, we associate a linear structure on $\MLS$\footnote{More precisely, the linear structure is defined on $\MLS\cup\{0\}$.} by pulling back the linear structure on the cotangent space $T^*_X\TS$. However, the induced  linear structure  depends on  $X$ (see Theorem \ref{thm:linearity:rigidity}).

\subsection{Main results}

%\begin{definition}
%  A surjective $\R$-linear isometry
 % \begin{equation*}
%    \Phi^*:(T^*_X\TS,\|\bullet\|_{\mathrm{Th}})
%  \longrightarrow(T^*_Y\TS,\|\bullet\|_{\mathrm{Th}})
%  \end{equation*}
%   is said to be \textit{geometric} if there exists an isometry
%  $\phi\in\MCGS$ such that $Y=\phi X$ and
%  $\Phi^*(d_X \length_\mu)=d_Y\length_{\phi\mu}$ for every $\mu\in\ML(X)$.
%\end{definition}
 %global section $d\ell_\mu: \TS\longrightarrow T^*\TS$. each isometry preserves $d\ell_\mu$ for all $\mu\in\MLS$.

 We answer Question \ref{question1} affirmatively.
\begin{theorem}
\label{thm:maintheorem1}
  Let $S_{g,n}$ be an orientable surface of finite  type with negative Euler characteristic. Then every  surjective $\R$-linear isometry
 \begin{equation*}
   \Phi^*:(T^*_X\TS,\|\bullet\|_{\mathrm{Th}})
  \to(T^*_Y\TS,\|\bullet\|_{\mathrm{Th}})
 \end{equation*}
  is induced by some isometry $\phi: X \to Y$ such that
  $\Phi^*(d_X \length_\mu)=d_Y\length_{\phi\mu}$ for every $\mu\in\ML(S_{g,n})$. In particular, $Y$ belongs to the $\MCGS$ orbit of $X$.
\end{theorem} 

We can also use the Thurston infinitesimal norm to distinguish the topology of the underlying hyperbolic surfaces. %The global version was proved by Walsh in \cite{Walsh2014}.
\begin{theorem}\label{thm:topology}
 Let $S_{g,n}$ and $S_{g',n'}$ be two   orientable surfaces of finite type with negative Euler characteristic.
If  there exists a surjective $\R$-linear isometry
 \begin{equation*}
  \Phi:(T_X\TS,\|\bullet\|_{\mathrm{Th}})
  \to(T_Y\TSprime,\|\bullet\|_{\mathrm{Th}})
\end{equation*}
 for some $X\in \TS$ and $Y\in \TSprime$, then $(g,n)=(g',n')$.
\end{theorem}
%\begin{remark}
 % For $S_{1,1}$, Dumas-Lenzhen-Rafi-Tao \cite{DLRT2016} showed that $T^*_X\T(S_{1,1}),\|\bullet\|_{\mathrm{Th}})$ and $(T^*_Y\T(S_{1,1}),\|\bullet\|_{\mathrm{Th}})$ are isometric if and only if  $Y$ belongs to the $\MCGS$ orbit of $X$.
%\end{remark}

Based on the infinitesimal rigidities above, we obtain both the local rigidity and the global rigidity  of the Thurston metric, as well as an analogue of the Patterson's theorem.

%\begin{remark}
%  A non-exceptional hyperbolic surface is a hyperbolic surface of genus $g$ with $n$ punctures such that $(g,n)\neq (0,4),(0,5),(0,6),(1,1),(1,2),(2,0)$.
%\end{remark}

\begin{theorem}[Local rigidity]\label{thm:localrigidity}
 Let $S_{g,n}$ be an orientable surface of finite type with negative Euler characteristic.  Let $U$ be a connected open set in $\TS$. Then any isometric embedding $(U,\dth)\to (\TS,\dth)$ is the restriction to $U$ of some element of $\MCGS$.
\end{theorem}
\begin{theorem}[Global rigidity]\label{thm:isometry:group}
 Let $S_{g,n}$ be an orientable surface of finite type with negative Euler characteristic.  Then every isometry of
  $\TS$ is an element of $\MCGS$. In particular,
  \begin{equation*}
    \Isom(\TS,\dth)=\left\{
    \begin{array}{ll}
    \MCGS/_{\Z_2}, &\text{if }  (g,n)\in\{(2,0), (1,1),(1,2)\},\\
    \MCGS/_{\Z_2\oplus\Z_2}, &\text{if }  (g,n)=(0,4),\\
      \MCGS, &\text{otherwise.}
    \end{array}
    \right.
  \end{equation*}
\end{theorem}
%\begin{remark}
%   As mentioned earlier, for $(g,n)\neq (0,4),(1,1),(1,2)$, Theorem \ref{thm:isometry:group} was first proved by Walsh \cite{Walsh2014}.
% \end{remark}

%Therefore, we have
\begin{theorem}[Topological rigidity]\label{thm:topology:global}
  Let $S_{g,n}$ and $S_{g',n'}$ be two   orientable surfaces of finite type with negative Euler characteristics. Then $(\TS,\dth)$ and $(\TSprime,\dth)$ are isometric if and only if $(g,n)=(g',n')$.
\end{theorem}

\begin{remark}
\begin{itemize}
 \item  Theorem \ref{thm:maintheorem1} and  Theorem \ref{thm:localrigidity}  were proved By Dumas-Lenzhen-Rafi-Tao \cite{DLRT2016} if $(g,n)=(1,1)$.
 \item Theorem \ref{thm:isometry:group} was proved Dumas-Lenzhen-Rafi-Tao \cite{DLRT2016} if $(g,n)=(1,1)$, and by Walsh \cite{Walsh2014} if $(g,n)\notin\{(0,4),(1,1),(1,2)\}$.
  \item  Theorem \ref{thm:topology:global} was first proved by Walsh \cite{Walsh2014} if $\{(g,n),(g'n')\}$ is not one of the following pairs:
  $ \{(0,4),(1,1)\}, \{(0,5),(1,2)\}, \{(0,6),(2,0)\}. $
  \item For surfaces of low complexity, Theorem \ref{thm:topology:global} is different from the case of Teichm\"uller metric. For the Teichm\"uller metric, Patterson \cite{Patterson1972} showed that $\TS$ and $\TSprime$ are isometric if and only if $(g,n)=(g',n')$ or $\{(g,n),(g'n')\}$ is one of the following pairs:
  $$ \{(0,4),(1,1)\}, \{(0,5),(1,2)\}, \{(0,6),(2,0)\}. $$
  \item Recently, Theorem \ref{thm:maintheorem1}-\ref{thm:topology:global} are proved independently by Huang-Ohshika-Papadopoulos \cite{HOP2021} using a differently method.
\end{itemize}
 
\end{remark}

\begin{remark}
  For more research about the Thuston metric, we refer to
    \cite{Thurston1998}, \cite{Liu2000}, \cite{PapadopoulosTheret2007}, \cite{ChoiRafi2007}, \cite{Theret2007}, \cite{PapadopoulosTheret2007a}, \cite{LenzhenRafiTao2012}, \cite{PapadopoulosTheret2012}, \cite{LPST2013}, \cite{Walsh2014}, \cite{PapadopoulosSu2014}, \cite{Theret2014} \cite{LenzhenRafiTao2015}, \cite{PapadopoulosSu2015}, \cite{PapadopoulosSu2016}, \cite{PapadopoulosYamada2017}, \cite{AlessandriniDisarlo2019}, \cite{HuangPapadopoulos2019},  \cite{DLRT2016}, and \cite{HOP2021}.
\end{remark}

 \subsection{Idea of proof}
 Our proof is based on the linear structures on $\MLS$ induced from the cotangent spaces of the Teichm\"uller space via Theorem \ref{thm:convex:model}. We show that the induced linear structures depends dramatically on the underlying hyperbolic surfaces (see Theorem \ref{thm:linearity:rigidity}).  By analysing the maximal flats (facets) in the unit sphere of the tangent space, we show that every surjective $\R$-linear isometry between tangent spaces induces an isomorphism of the curve complexes for non-exceptional surfaces (see Proposition \ref{prop:isometry:isomorphism}).  The linearity rigidity of linear structures on $\MLS$ and   the automorphism rigidity of the curve complex then enable us to prove Theorem  \ref{thm:maintheorem1} and Theorem \ref{thm:topology} in this case. For exceptional surfaces, we adapt the idea from \cite{DLRT2016}.  The proofs of Theorem \ref{thm:localrigidity}, Theorem \ref{thm:isometry:group} and Theorem  \ref{thm:topology:global}  are standard assuming Theorem \ref{thm:regularity} and Theorem \ref{thm:maintheorem1}.

\subsection*{Acknowledgements}
I am grateful to Kasra Rafi for sharing with me the question about the local rigidity of  Thurston metric. I thank Weixu Su for his comments. Finally, I would like to thank the referee for his/her comments and suggestions.  This work is supported by NSFC 11901241.

\section{Preliminaries}
%Through out this paper, we assume that $S_{g,n}$ is not a sphere with four or fewer punctures, nor a torus with two punctures.
 %In this section, we collect some preliminaries.

\subsection{Teichm\"uller space and measured laminations} Continuing with the same notation as in the introduction,  let $S_{g,n}$ be an oriented surface of genus $g$ with $n$ punctures whose Euler characteristic is negative: $\chi(S_{g,n})=2-2g-n<0$.
A \textit{marked hyperbolic surface} is a pair $(X,f)$, where $X$ is a hyperbolic surface and $f:S_{g,n}\to X$ is an orientation-preserving homeomorphism. Two marked hyperbolic surfaces $(X,f)$ and $(X',f')$ are said to be \textit{equivalent} if there exists an isometry $X\to X'$ which is isotopic to $f'\circ f^{-1}$. The \textit{Teichm\"uller space} $\TS$ is the set of equivalence classes of marked hyperbolic surfaces. Sometimes,  we simply denote an equivalence class by $X$. With the topology induced by the Thurston metric (defined in (\ref{eq:thurston:metric})),   $\TS$ is homeomorphic to  $\R^{6g-6+2n}$.

A \textit{measured lamination} is a lamination on $S_{g,n}$ equipped with a transverse invariant measure.   Every measured lamination $\mu$ induces a functional over $\C^0(S_{g,n})$, the set of isotopy classes of simple closed curves on $S_{g,n}$, by associating to each $\alpha\in \C^0(S_{g,n})$ the \textit{intersection number} $\i(\mu,\alpha)$:
 \begin{equation*}
   \i(\mu,\alpha):=\inf_{\alpha'}\int_{\alpha'}d\mu
 \end{equation*}
 where $\alpha'$ ranges over all simple closed curves $\alpha'$ isotopic to $\alpha$.
 Two measured laminations are said to be \textit{equivalent} if they induce the same functional on  $\C^0(S_{g,n})$.  Let $\MLS$ be the space of equivalence classes of measured laminations equipped with the weak-* topology induced from the space of functionals over $\C^0(S_{g,n})$. Let $\PMLS:=\MLS/_{\R_{>0}}$ be the space of  projective classes of measured laminations. Thurston showed that $\MLS$ is homeomorphic to $\R^{6g-6+2n}$, and that the set of weighted simple closed curves is dense in $\MLS$ (see \cite{FLP} and \cite{CassonBleiler1988} for more details).

\subsection{Mapping class group action}
Let $\mathrm{Homeo}^\pm(S_{g,n})$ (resp.  $\mathrm{Homeo}^+(S_{g,n})$) be the group of homeomorphisms (resp. orientation-preserving homeomorphisms) of $S_{g,n}$. Let $\mathrm{Homeo}_0(S_{g,n})$ be the subgroup of homeomorphisms isotopic to the identity. Then the mapping class group $\MCG^+(S_{g,n})$ of $S_{g,n}$ is defined to be:
\begin{equation*}
  \MCG^+(S_{g,n})=\mathrm{Homeo}^+(S_{g,n})/
  \mathrm{Homeo}_0(S_{g,n})
\end{equation*}
and the extended mapping class group
$\MCGS$ is defined to be:
\begin{equation*}
  \MCGS=\mathrm{Homeo}^\pm(S_{g,n})/
  \mathrm{Homeo}_0(S_{g,n}).
\end{equation*}
The extended mapping class group acts on $\TS$ as follows. Each orientation-preserving mapping class $\phi\in\MCG^+{S_{g,n}}$ acts on $\TS$ by pre-composition:
\begin{equation*}
  \phi\cdot (X,f):=(X,f\circ \phi^{-1}).
\end{equation*}
We now consider the action of orientation-reversing homeomorphisms. For every hyperbolic surface $X$, let $\overline{X}$ be $X$ with the opposite orientation. Let $j_X:X\to\overline X$ be the identity map. In particular, $j_X$ is an orientation-reversing isometry. Now the image of $(X,f)$ under $\psi$ is defined to be:
\begin{equation*}
  \psi\cdot (X,f):=(\overline X,j_X\circ f\circ \psi^{-1}).
\end{equation*} 

 By definition, we see that $\MCGS$ acts on $(\TS,\dth)$ by isometries.

 \subsection{Stretch lines}
 A geodesic lamination on $X\in \TS$ is a lamination whose leaves are hyperbolic geodesics. A geodesic lamination is said to be \textit{maximal} if each component of $X-\lambda$ is an ideal hyperbolic triangle.

 \begin{theorem}[\cite{Thurston1998}, Corollary 4.2]
     For any complete hyperbolic structure $X$ of finite area on a surface $S$, for any maximal geodesic lamination $\lambda$ not all of whose leaves go to a cusp at both ends there is a new hyperbolic structure
\[
\operatorname{stretch}(X, \lambda, t)
\]
depending analytically on $t>0$ such that
\begin{enumerate}[(a)]
  \item the identity map is Lipschitz with Lipschitz constant
\[
L(X, \operatorname{stretch}(X, \lambda, t))= \exp (t)
\]

\item the identity map exactly expands arc length of $\lambda$ by the constant factor $\exp (t)$, while for each point on $X-\lambda$   the Lipschitz constant is strictly less than $\exp(t)$.
\end{enumerate}
 \end{theorem}

 A direct consequence is the following infinitesimal version.
 \begin{corollary}\label{cor:infinitesimal}
   Let $v\in T_X\TS$ be the vector tangent to the stretch line $\mathrm{stretch}(X,\lambda,t)$. Then
   \begin{equation*}
     (d_X \log\ell_\mu)(v)\leq 1, \forall \mu\in\MLS
   \end{equation*}
   and
   \begin{equation*}
     (d_X \log\ell_\mu)(v)= 1\iff \mathrm{supp}\mu\subset\lambda
   \end{equation*}
   where $\mathrm{supp} \mu$ represents the supporting lamination of $\mu$.
 \end{corollary}
 Using stretch lines, Thurston (\cite[Theorem 8.5]{Thurston1998}) showed that for every pair of points in $\TS$, there exists a geodesic with respect to $\dth$, which is a concatenation of stretch lines. In general, such geodesics are not unique. An interesting question is to describe the set of the union of all geodesics between two points, see \cite{DLRT2016} for the case of punctured torus.

\subsection{Curve complex}
 For $(g,n)\neq (0,4),(1,1)$,  the \textit{curve complex} $\C(S_{g,n})$ is a complex whose vertices are isotopy classes of simple closed curves on  $S_{g,n}$ and  whose $k$-simplices are collections of $k+1$  isotopy classes of simple closed curves on $S_{g,n}$ which can be realized disjointly.   For our purpose, we need the following connectedness and rigidity of  $\C(S_{g,n})$.
 \begin{theorem}[\cite{Harvey1981}, see also \cite{Putman2008}]\label{thm:connectedness}
   $\C(S_{g,n})$ is connected.
 \end{theorem}

 \begin{theorem}[\cite{Ivanov1997}, \cite{Korkmaz1999}, \cite{Luo2000}]
 \label{thm:curvecomplex:rigidity}
 (a) Suppose that $\{(g,n),(g',n')\}$ is not one of the following three pairs:
$
  \{(0,4),(1,1)\}, \{(0,5),(1,2)\},  \{(0,6),(2,0)\}.
$
then $\C(S_{g,n})$ and $\C(S_{g',n'})$ are isomorphic if and only if $(g,n)=(g',n')$.

 (b) Suppose that  $S_{g,n}$ is not a sphere with four or fewer punctures, nor a torus with two or fewer punctures. Then every automorphism of $\C(S_{g,n})$ is an element of $\MCGS$, and \begin{equation*}
         \mathrm{Aut}(\C(S_{g,n})):=\left\{
         \begin{array}{ll}
           \MCGS/_{\Z_2}, & \text{if } (g,n)=(2,0) \\
           \MCGS, & \text{otherwise}
         \end{array}
         \right.
       \end{equation*}
 Moreover, any automorphism of $\C(S_{1,2})$  preserving the set of vertices represented by separating curves is
induced by an extended mapping class.
 \end{theorem}

We also need the curve complexes of surfaces with low complexity. The curve complex $\C(S_{0,4})$ (resp. $\C(S_{1,1})$) is a graph whose vertices are isotopy classes of simple closed curves on  $S_{0,4}$ (resp. $S_{1,1}$) and  whose edges are pairs of  isotopy classes of simple closed curves with geometric intersection number one (resp. two).

For simple closed curves $\alpha$ and $\beta$, we denote by $D^n_\alpha\beta$ the $n-$th Dehn twist of $\beta$ along $\alpha$. We say that an isomorphism $\iota:\C(S_{g,n})\to\C(S_{g',n'})$ is  \textit{compatible with  Dehn twists} if $\iota(D_\alpha\beta)=D_{\iota(\alpha)}\iota(\beta)$ holds for   any simple closed curves $\alpha$ and $\beta$.

\begin{lemma}[\cite{Luo2000}, Lemma 2.1] \label{lem:patterson}
Suppose that  $\{(g,n),(g',n')\}$ is  one of the following three pairs:
$
  \{(0,4),(1,1)\}, \{(0,5),(1,2)\},  \{(0,6),(2,0)\}.
$
then $\C(S_{g,n})\simeq\C(S_{g',n'})$
  Moreover any isomorphism $\C(S_{g,n})\to\C(S_{g',n'})$   is compatible with the Dehn twists.
\end{lemma}

\begin{proof}[Sketch of the proof]
  Let $r:S_{2,0}\to S_{2,0}$ be a hyperelliptic involution. Then $r$ fixes every simple closed curve. The quotient map $\pi: S_{2,0}\to S_{2,0}/_\sim$ induces an isomorphism $\pi^*:\C(S_{2,0})\to\C(S_{0,6})$.
  Any other isomorphism is  a precomposition of $\pi^*$ with some automorphism of $\C(S_{2,0})$.
   The compatibility with Dehn twists then follows from the automorphism rigidity of $\C(S_{2,0})$ in Theorem \ref{thm:curvecomplex:rigidity}.  By considering subgraphs of  $\C(S_{2,0})$ and $\C(S_{0,6})$, one gets isomorphisms $\C(S_{1,2})\to\C(S_{0,5})$ and $\C(S_{1,1})\to\C(S_{0,4})$ as well as the compatibility.
\end{proof}

\begin{remark}\label{rmk:torus:sphere}
   Let $\iota$ be either an automorphism of $(\C(S_{1,2}))$ or   an isomorphism from  $(\C(S_{1,2}))$ to $(\C(S_{0,5}))$ which sends some non-separating curve $\alpha$ to a separating curve $\iota(\alpha)$. Considering the subgraphs induced by the curves disjoint from $\alpha$ and $\iota(\alpha)$ respectively, we see that  $\iota$ induces an isomorphism $\iota^*:\C(S_{0,4})\to\C(S_{1,1})$.
\end{remark}

\subsection{Exceptional surface $S_{1,1}$}
%In this section, we apply the idea of \cite{DLRT2016} to prove Theorem \ref{thm:maintheorem1} for $S_{0,4}$ and Theorem \ref{thm:topology} for $\{S_{0,4},S_{1,1}\}$.
% \subsection{Once-punctured torus $S_{1,1}$}
We  recall some results from \cite{DLRT2016} in the case of punctured torus $S_{1,1}$.
Notice that for both $S_{0,4}$ and $S_{1,1}$, the Teichm\"uller spaces are of dimension two. Accordingly, the unit tangent spheres are of dimension one.
Let $\alpha$ be a simple closed curve.  Let
$$ F_X(\alpha):=\{v\in T^1_X\TS:~(d_X \log\ell_\alpha)(v)=1 \} $$
be the facet in the unit tangent sphere $T^1_X\TS$ corresponding to $\alpha$. Let $|F_X(\alpha)|$ be the length of $F_X(\alpha)$ with respect to the Thurston norm $\|\cdot\|_\mathrm{Th}$.

\begin{proposition}[\cite{DLRT2016}, Proposition 6.7]\label{prop:facet:length1} Let $X\in\T(S_{1,1})$.
  There exists a constant $C_1=C_1(X)$ depending on $X$, such that for every simple closed curve $\alpha$,
  \begin{equation*}
    \frac{1}{C_1}\ell_\alpha(X)^2e^{-\ell_\alpha(X)}
    \leq |F_X(\alpha)|\leq C_1 \ell_\alpha(X)^2e^{-\ell_\alpha(X)}.
  \end{equation*}
\end{proposition}
This implies
\begin{theorem}[\cite{DLRT2016}, Theorem 6.8]\label{thm:DLRT:facet:limit}
  Let $X\in\T(S_{1,1})$. Let $\alpha$ and $\beta$ be two simple closed curves with $\i(\alpha,\beta)=1$. Let $\beta_n=D^n_\alpha\beta$. Then
  \begin{equation*}
    \lim_{n\to\infty}\frac{|\log |F_X(\beta_n) ||}{|n|}=\ell_\alpha(X).
  \end{equation*}
\end{theorem}

 For each simple closed curve $\alpha$ on $S_{0,4}$ or $S_{1,1}$, there are two  canonical maximal laminations $\alpha^+$ and $\alpha^-$\footnote{\small For $S_{1,1}$, $\alpha^+\backslash\alpha$ (resp. $\alpha^-\backslash\alpha$) consists of three bi-infinite simple arcs $\{\delta,\eta_1,\eta_2\}$ such that $\delta$ spirals to the left (resp. right) at both ends around $\alpha$ while $\eta_i$ spirals to left (resp. right) around $\alpha$ at one end and approaches to the puncture at the other end. For $S_{0,4}$, $\alpha^+\backslash\alpha$ (resp. $\alpha^-\backslash\alpha$) consists of six bi-infinite simple arcs $\{\delta_1,\delta_2,\eta_1,\eta_2,\eta_3,\eta_4\}$ such that $\delta_i$ spirals to the left (resp. right) at both ends around $\alpha$ while $\eta_i$ spirals to left (resp. right) around $\alpha$ at one end and approaches to a puncture at the other end.}.
Let $v_X(\alpha^+)$ and $v_X(\alpha^-)$ be respectively the unit tangent vectors to the stretch paths $\mathrm{ stretch}(X,\alpha^+,t)$ and $\mathrm{stretch}(X,\alpha^-,t)$.
Then  $v_X(\alpha^+)$ and $v_X(\alpha^-)$ are the endpoints of $F_X(\alpha)$. Hence
\begin{equation*}
  |F_X(\alpha)|=\| v_X(\alpha^+)-v_X(\alpha^-)\|_{\mathrm{Th}}.
\end{equation*}
Fix a simple closed curve $\alpha$.  For
 $\gamma\neq\alpha$, let $(v_X(\gamma^-),v_X(\alpha^+))$ (resp. $(v_X(\gamma^+),v_X(\alpha^-))$) be the subinterval of  $T^1_X\T(S_{1,1})\setminus F_X(\alpha)$  whose  endpoints are  $v_X(\gamma^-)$  and $v_X(\alpha^+)$ (resp.  $v_X(\gamma^+)$  and $v_X(\alpha^-)$).

Combining Theorem \ref{thm:DLRT:facet:limit} and the train track approximate, Dumas-Lenzhen-Rafi-Tao also proved the following, which is contained in the proof of \cite[Theorem 1.4]{DLRT2016} and which we summarize as a proposition.
\begin{proposition}[\cite{DLRT2016}]
\label{prop:correspondence1}
 Let $X\in\T(S_{1,1})$. There exits a constant $L_1$ depending on $X$ such  that for any simple closed curve $\alpha$ with $\ell_\alpha(X)>L_1$,  there exist a simple closed curve $\beta$ with $\i(\alpha,\beta)=1$ and a positive integer $N$, such that
 \begin{itemize}
   \item for each $n>N$, $F_X(\beta_n)$ is the longest facet in $(v_X(\beta_{n-1}^-),v_X(\alpha^+))$,  and
   \item for each $n<-N$, $F_X(\beta_n)$ is the longest facet in $(v_X(\beta_{n+1}^+),v_X(\alpha^-))$,
 \end{itemize}
 where $\beta_n=D^n_\alpha\beta$.
\end{proposition}

\subsection{Exceptional surface $S_{0,4}$}
Using the same idea as in \cite{DLRT2016}, one can prove analogues for $S_{0,4}$. Namely
\begin{proposition} \label{prop:facet:length}
Let $X\in\T(S_{0,4})$.
  There exists a constant $C_2=C_2(X)$ depending on $X$, such that for every simple closed curve $\alpha$,
  \begin{equation*}
    \frac{1}{C_2}\ell_\alpha(X)^2e^{-\ell_\alpha(X)/2}
    \leq |F_X(\alpha)|\leq C_2 \ell_\alpha(X)^2e^{-\ell_\alpha(X)/2}.
  \end{equation*}
\end{proposition}
\begin{theorem}\label{thm:facet:length:limit}
  Let $X\in\T(S_{0,4})$. Let $\alpha$ and $\beta$ be two simple closed curves with $\i(\alpha,\beta)=2$. Let $\beta_n=D^n_\alpha\beta$. Then
  \begin{equation*}
    \lim_{n\to\infty}\frac{|\log |F_X(\beta_n) ||}{|n|}=\ell_\alpha(X).
  \end{equation*}
\end{theorem}

\begin{proposition}\label{prop:correspondence}
 Let $X\in\T(S_{0,4})$. There exits a constant $L_2$ depending on $X$ such  that for any simple closed curve $\alpha$ with $\ell_\alpha(X)>L_2$,  there exist a simple closed curve $\beta$ with $\i(\alpha,\beta)=2$ and a positive integer $N$, such that
 \begin{itemize}
   \item for each $n>N$, $F_X(\beta_n)$ is the longest facet in $(v_X(\beta_{n-1}^-),v_X(\alpha^+))$,  and
   \item for each $n<-N$, $F_X(\beta_n)$ is the longest facet in $(v_X(\beta_{n+1}^+),v_X(\alpha^-))$,
 \end{itemize}
 where $\beta_n=D^n_\alpha\beta$.
\end{proposition}
\begin{remark}
 Theorem \ref{thm:facet:length:limit} is a direct consequence of Proposition \ref{prop:facet:length}. The proof of Proposition \ref{prop:correspondence} is exactly the same as that of Proposition \ref{prop:correspondence1} except that  one replace Proposition \ref{prop:facet:length1} and Theorem \ref{thm:DLRT:facet:limit} by Proposition \ref{prop:facet:length} and Theorem \ref{thm:facet:length:limit}.
  Proposition \ref{prop:facet:length} is a slightly different from Proposition \ref{prop:correspondence1} in the sense that $|F_X(\alpha)|$ is comparable to $\ell_\alpha(X)^2e^{-\ell_\alpha(X)/2}$ in Proposition \ref{prop:facet:length} while it is comparable to $\ell_\alpha(X)^2e^{-\ell_\alpha(X)}$ in Proposition \ref{prop:facet:length1}. The difference is due to the calculation of Fenchel-Nielsen coordinates. Namely, in this case the equation in the fourth line in \cite[Page 50]{DLRT2016} is replaced by the following:
  $$\Delta(t)=\tau_\alpha(X^+_t)-\tau_\alpha(X^-_t)= 4e^t\log \frac{e^{\ell_\alpha(X)}+1}{e^{\ell_\alpha(X)/2}-1}
  -4\log \frac{e^{e^t\ell_\alpha(X)}+1}{e^{e^t\ell_\alpha(X)/2}-1}.$$
The calculation of this equation is almost the same as the one demonstrated in the proof of \cite[Theorem 5.1]{DLRT2016}. We omit the details.
\end{remark}

 %================================================
 \section{Linear structures for the space of measured laminations}\label{sec:linear:structure}
 %================================================

By identifying $\PMLS$ with $\{\mu\in\MLS:\ell_\mu(X)=1\}$, we obtain a homeomorphism
  \begin{equation*}
     \begin{array}{cccc}
       \D_X: & \MLS & \longrightarrow & T^*_X\TS\\
        & \mu & \longmapsto & d_X \ell_\mu.
     \end{array}
   \end{equation*}
   In this way, we associate a linear structure to $\MLS$ by pulling back the linear structure on the cotangent space $T^*_X\TS$.
 % Recall that the map
 % \begin{equation*}
 %    \begin{array}{cccc}
 %      \D_X: & \MLS& \longrightarrow & T^*_X\TS-\{0\}\\
 %       & \mu & \longmapsto & d_X \ell_\mu
 %    \end{array}
 %  \end{equation*}
 %   induces a linear structure on $\MLS$ by pulling back the linear structure on the cotangent space $T^*_X\TS$. %The induced linear structure depends on the  underlying hyperbolic surface $X$.

   Notice that for each $\mu\in \MLS$, we have
   \begin{equation*}
     \|d_X \log \ell_\mu\|_{\mathrm{Th},X}=1.
   \end{equation*}
   Therefore,
   \begin{equation*}
     \|d_X \ell_\mu\|_{\mathrm{Th},X}=\ell_\mu(X), ~\forall \mu\in\MLS, ~\forall X\in \TS.
   \end{equation*}

  % By declaring the image of $\PML(S)$ under the map $\PD_X$ in $T^*_X(\T(S))$ to be the unit sphere,  Thurston defines a nonsymmetric norm on the cotangent space $T^*_X(\T(S))$. Dually, this defines a nonsymmetric norm on the tangent space $T_X(\T(S))$:
 %  \begin{equation}\label{eq:thurston:norm}
 %    \|v\|_{\mathrm{Th}}:=\sup _{\lambda \in \PML(S)}
 %    (d \log\length_\lambda)[v]=\sup _{\lambda \in \ML(S)}
 %    \frac{ v( \length_\lambda)}{\length_\lambda (X)}.
 %  \end{equation}

   For  two marked hyperbolic surfaces $X,Y\in \TS$, let us consider a homeomorphism $\Gamma_{X,Y}:T_X^*\TS\longrightarrow
     T^*_Y\TS$ which is the unique continuous extension of
   \begin{equation}\label{eq:homeo:tangent:XY}
   \begin{array}{cccc}
       \Gamma_{X,Y}:& T_X^*\TS-\{0\}&\longrightarrow&
     T^*_Y\TS-\{0\}\\
     & d_X\ell_\mu&\longmapsto &
      \frac{\ell_\mu(X)}{\ell_\mu(Y)}d_Y\ell_\mu.
   \end{array}
   \end{equation}
   Therefore,
    \begin{equation*}
     \|\Gamma_{X,Y}(d_X\ell_\mu)\|_{\mathrm{Th},Y}=
  \|d_X\ell_\mu\|_{\mathrm{Th},X}=\ell_\mu(X).
   \end{equation*}
   In particular, $\Gamma_{X,Y}(0)=0$.

   In general, $\Gamma_{X,Y}$ is not linear. In fact, we have the following rigidity.

  \begin{theorem}[Linearity rigidity]\label{thm:linearity:rigidity}
  Suppose that $S_{g,n}$ is not a sphere with four or fewer punctures, nor a torus with one or fewer punctures.
  Let $X,Y\in \T(S_{g,n})$. Then
    $\Gamma_{X,Y}$ is linear if and only if $X=Y$.
  \end{theorem}
  \begin{proof}
   If $X=Y$, then  $\Gamma_{X,Y}$ is the identity map, which is linear. We now consider the converse. In the following, we assume  that  $\Gamma_{X,Y}$ is linear.

    Let $\alpha$ and $\beta$ be two disjoint,  non-isotopic  simple closed curves.  Then $\alpha+\beta$ is also a measured lamination.
    Therefore,
    \begin{eqnarray*}
      \Gamma_{X,Y}(d_X\ell_\alpha+d_X\ell_\beta)&=&
      \Gamma_{X,Y}(d_X(\ell_{\alpha+\beta}))\\
       &{=}&
      \frac{\ell_{\alpha+\beta}(X)}
      {\ell_{\alpha+\beta}(Y)}
      d_Y\ell_{\alpha+\beta} \qquad (\text{by } (\ref{eq:homeo:tangent:XY}))\\
      &{=}&
      \frac{\ell_\alpha(X)+\ell_\beta(X)}
      {\ell_\alpha(Y)+\ell_\beta(Y)}
      (d_Y\ell_\alpha+d_Y\ell_\beta).
    \end{eqnarray*}
    On the other hand, it follows from the linearity of $\Gamma_{X,Y}$ that
    \begin{eqnarray*}
     \Gamma_{X,Y}(d_X\ell_\alpha+d_X\ell_\beta)&=&
     \Gamma_{X,Y}(d_X\ell_\alpha)+\Gamma_{X,Y}
     (d_X\ell_\beta)\\
     &=& \frac{\ell_\alpha(X)}{\ell_\alpha(Y)}
     d_Y\ell_\alpha+
     \frac{\ell_\beta(X)}{\ell_\beta(Y)}
     d_Y\ell_\beta.
    \end{eqnarray*}
    By Theorem \ref{thm:convex:model}, we see that  $ d_Y\ell_\alpha$ and $ d_Y\ell_\beta$ are linearly independent in $T^*\TS$.
    Comparing the two equations above, we have
    \begin{equation*}
       \frac{\ell_\alpha(X)}{\ell_\alpha(Y)}=
        \frac{\ell_\alpha(X)+\ell_\beta(X)}
      {\ell_\alpha(Y)+\ell_\beta(Y)}=
       \frac{\ell_\beta(X)}{\ell_\beta(Y)}.
    \end{equation*}

    Now, Let   $\alpha,\delta$ be an arbitrary pair of simple closed curves.  By the connectedness of the curve complex (Theorem \ref{thm:connectedness}), there exists a sequence of simple closed curves $\alpha_0=\alpha,\alpha_1,\cdots,\alpha_k=\delta$, such that $\alpha_i$ and $\alpha_{i+1}$ are disjoint for each $i=0,1,\cdots, k-1$.
    It then follows from the discussion above that
    \begin{equation*}
      \frac{\length_{\alpha_0}(X)}{\length_{\alpha_0}(Y)}
      =\frac{\length_{\alpha_1}(X)}{\length_{\alpha_1}(Y)}
      =\cdots=
      \frac{\length_{\alpha_k}(X)}{\length_{\alpha_k}(Y)}.
    \end{equation*}
    In particular,
    \begin{equation*}
      \frac{\length_{\alpha}(X)}{\length_{\alpha}(Y)}=
      \frac{\length_{\delta}(X)}{\length_{\delta}(Y)}.
    \end{equation*}
    By the arbitrariness of $\alpha$ and $\delta$,  we see that there exists a constant $K$, such that
    \begin{equation*}
      \frac{\length_{\gamma}(X)}{\length_{\gamma}(Y)} \equiv K, ~\forall \gamma\in \C^0(S_{g,n}).
    \end{equation*}
    This implies that $K=1$. (Otherwise, $\d_{\mathrm{Th}}(Y,X)<0$ if $K<1$, or $\d_{\mathrm{Th}}(X,Y)<0$ if $K>1$.)
    Therefore, $Y=X$.

  \end{proof}
  %\begin{remark}
  %  There is  another natural map
  %  $ F_{X,Y}:T_X^*\TS\longrightarrow
  %   T^*_Y\TS$ which is the unique continuous extension of the composition of
  %   $\D_X^{-1}:T_X^*\TS-\{0\}\longrightarrow
  %   \MLS$ and
  %   $ \D_Y:\MLS\longrightarrow
  %   T^*_X\TS-\{0\}$. In particular,
  %  $F_{X,Y}(h_\mu)=d\log\length_\mu$.
  %  By (\ref{eq:homeo:tangent:XY}), it follows that
  %  \begin{equation*}
   %  F_{X,Y}=\Gamma_{X,Y} \iff Y=X.
   % \end{equation*}
  %\end{remark}

  \section{Flats in the unit tangent spheres}
  %\subsection{Flats in the unit tangent spheres}
  There are flat places on the unit sphere $T^1_X\TS$ of $T_X\TS$.  A \textit{facet} is a maximal flat portion of $T^1_X\TS$  which has maximum possible dimension so that it has interior.

  \begin{theorem}[\cite{Thurston1998},Theorem 10.1]
    \label{thm:mostly:facets}
   There is a bijection between the set of facets on the unit sphere  $T^1_X\TS$ and the set of simple closed curves. In other words, every facet is contained in a plane $d_X \log\ell_\alpha=1$ for some simple closed curve $\alpha$.
  \end{theorem}
  For each simple closed curve $\alpha$, let
  $$ F_X(\alpha):=\{v\in T^1_X\TS:~(d_X \log\ell_\alpha)(v)=1 \} $$
  be the corresponding facet  obtained in Theorem \ref{thm:mostly:facets}.

  \begin{lemma}\label{lem:disjointness}
    Let $\alpha,\beta\in \C^0(S_{g,n})$ be two simple closed curves. Then
    \begin{equation*}
     \i(\alpha,\beta)=0 \iff \partial \mathrm{F}_X(\alpha) \cap \partial \mathrm{F}_X(\beta)\neq \emptyset.
    \end{equation*}
  \end{lemma}
  \begin{proof}
    (i) Assume that $\i(\alpha,\beta)=0$. Let $\lambda$ be a maximal geodesic lamination which contains both $\alpha$ and $\beta$. Let $v$ be the vector tangent to the stretch path $\mathrm{stretch}(X,\lambda,t)$. Let $X_t:=\mathrm{stretch}(X,\lambda,t)$.  Then
     \begin{equation*}
      \frac{ \length_\beta(X_t)}{\length_\beta(X)}
      =\frac{ \length_\alpha(X_t)}{\length_\alpha(X)}=e^t,
     \end{equation*}
     which implies that $(d_X \log\ell_\alpha)(v)=
     (d_X \log\ell_\beta)(v)=1$. In particular,
     \begin{equation*}
       v\in \partial \mathrm{F}_X(\alpha) \cap \partial \mathrm{F}_X(\beta)\neq \emptyset.
     \end{equation*}
     \bigskip

    (ii) Assume that $\partial \mathrm{F}_X(\alpha) \cap \partial \mathrm{F}_X(\beta)\neq \emptyset.$
    Let $v\in \partial \mathrm{F}_X(\alpha) \cap \partial \mathrm{F}_X(\beta)$.  Then
    \begin{equation}\label{eq:alpha:beta}
     \|v\|_{\mathrm{Th}}= \sup_{\mu\in\ML(S)}~ (d_X \log\ell_\mu)(v)=
      (d_X \log\ell_\alpha)(v)= (d_X \log\ell_\beta)(v)=1.
    \end{equation}
    For each $0\leq s\leq  1$,
     \begin{eqnarray*}
       &&\|s\cdot d_X \log\ell_\alpha+(1-s)\cdot d_X \log \ell_\beta\|_{\Th}\\
       &=&\sup_{v'\in T^1_X\TS}(s\cdot d_X \log\ell_\alpha+(1-s)\cdot d_X\log \ell_\beta)(v')\\
       &\geq&(s\cdot d_X \log\ell_\alpha+(1-s)\cdot d_X\log \ell_\beta)(v)\\
       &=&1.  \qquad\qquad(\text{by} (\ref{eq:alpha:beta}))
     \end{eqnarray*}
     On the other hand,
     \begin{eqnarray*}
       \|s\cdot d_X \log\ell_\alpha+(1-s)\cdot d_X \log \ell_\beta\|_{\Th} &\leq& s\| d_X \log\ell_\alpha\|_{\Th}+(1-s)\|d_X\log\ell_\beta \|_{\Th}=1.
     \end{eqnarray*}
     Consequently,
     \begin{equation*}
       \|s\cdot d_X \log\ell_\alpha+(1-s)\cdot d_X \log \ell_\beta\|_{\Th} =1,~~\forall 0\leq s\leq 1.
     \end{equation*}
    In particular, the segment
     \begin{equation*}
       \{s\cdot d_X \log\ell_\alpha+(1-s)\cdot d_X \log\ell_\beta: 0\leq s\leq 1\}\subset T^*_X\TS
     \end{equation*}
     is contained  in the unit sphere of the cotangent space $T^*_X\TS$.
    It then follows from Theorem \ref{thm:convex:model} that  for each $0\leq s\leq 1$,  there exists a unique $\mu_s\in \PML(S)$ such that
    \begin{equation}\label{eq:ut}
      s\cdot d_X \log\length_\alpha+(1-s)\cdot d_X \log\length_\beta=d_X \log\length_{\mu_s}.
    \end{equation}
     Suppose to the contrary that $\i(\alpha,\beta)>0$. Let $\widehat{\mu_s}$ be a maximal geodesic lamination obtained from $\mu_s$ by adding isolated leaves. In other words, $\mu_s$ is the maximal measured geodesic lamination contained in $\widehat{\mu_s}$.  Let $v_s\in T_X\TS$ be the vector tangent to the stretch path $\mathrm{stretch}(X,\widehat{\mu_s},t)$ at $X$.  Then by Corollary \ref{cor:infinitesimal},
     \begin{equation}\label{eq:atmost1}
     (d_X \log\ell_\mu)(v_s)\leq 1, \forall \mu\in\MLS
   \end{equation}and
     \begin{equation}\label{eq:v:mu:s}
       (d_X \log\ell_{\mu} )(v_s)=1\iff \mu\subset\mu_s.
     \end{equation}
     On the other hand, since $\i(\alpha,\beta)>0$, it follows either $\alpha$ or $\beta$ is not contained in $\widehat{\mu_s}$. Consequently,
     \begin{equation*}
        (d_X \log\length_{\alpha})(v_s)<1
     \end{equation*}
     or
     \begin{equation*}
         (d_X \log\length_{\beta})(v_s)<1.
     \end{equation*}
    Combined with (\ref{eq:atmost1}), this implies that
         \begin{equation*}
      (s\cdot d_X \log\length_\alpha+(1-s)\cdot d_X \log\length_\beta)(v_s)<1=(d_X \log\length_{\mu_s})(v_s),~\forall 0<s<1.
    \end{equation*}
   This contradicts to (\ref{eq:ut}), which proves the lemma.
  \end{proof}

  \begin{proposition}\label{prop:isometry:isomorphism}
    Suppose that $S_{g,n}$  and $S_{g',n'}$  are not spheres with four or fewer punctures, nor  tori with one or fewer punctures. Then every surjective $\R$-linear isometry
 \begin{equation*}
  \Phi:(T_X\TS,\|\bullet\|_{\mathrm{Th}})
  \to(T_Y\TSprime,\|\bullet\|_{\mathrm{Th}})
\end{equation*}
induces an isomorphism $\widehat{\Phi}:\mathscr{C}(S_{g,n})\to\mathscr{C}(S_{g',n'})$ such that for every $\alpha\in\mathscr C^0(S_{g,n})$,
\begin{enumerate}[(a)]
  \item $ \Phi(F_X(\alpha)) =F_{Y}(\widehat{\Phi}(\alpha))$, and
  \item $\widehat{\Phi}(\alpha)$ is separating if and only if $\alpha$ is separating.
\end{enumerate}

\end{proposition}

\begin{proof}
By assumption, $\Phi$ induces a bijection between the set of facets of $T^1_X\TS$ and the set of faces of $T^1_Y\T(S_{g',n'})$.
  By Theorem \ref{thm:mostly:facets}, $\Phi$ also induces a bijection $\widehat{\Phi}$ between $\mathscr{C}^0(S_{g,n})$ and $\mathscr{C}^0(S_{g',n'})$. Moreover, it follows from Lemma \ref{lem:disjointness} that $\widehat{\Phi}$ preserves the disjointness of simple closed curves. This implies that $\widehat{\Phi}:\mathscr{C}(S_{g,n})\to\mathscr{C}(S_{g',n'})$ is a simplicial isomorphism. This proves the statement (a).

  Let us now consider (b). If neither $(g,n)$ nor $(g',n')$ equals $(1,2)$, the statement follows from Theorem \ref{thm:curvecomplex:rigidity} and Lemma \ref{lem:patterson}. Now assume that  one of $(g,n)$ and $(g',n')$ equals $(1,2)$, say $(g,n)$,  then by Lemma \ref{lem:patterson} $(g',n')=(1,2)$ or $(0,5)$.

  Notice that the map $\Phi$ induces an $\R$-linear isometry 
  \begin{equation*}
  \Gamma:(T^*_X\TS,\|\bullet\|_{\mathrm{Th}})
  \to(T^*_Y\TSprime,\|\bullet\|_{\mathrm{Th}})
\end{equation*}
so that $$\Gamma\left(\frac {d_X\ell _\alpha}{\ell_\alpha(X)}\right)=\frac{d_Y\ell_{\widehat{\Phi}(\alpha)}}{\ell_{\widehat\Phi(\alpha)}(Y)}$$ for every simple closed $\alpha$. Since the set of weighted simple closed curves is dense in $\MLS$, it follows that  the identity above also holds for any measured lamination.
   Applying an argument  similar to the proof of Theorem \ref{thm:linearity:rigidity},  we see that
  there exists a positive constant $K$ such  that  $~\forall \alpha\in \C^0(S_{g,n})$
\begin{equation}\label{eq:ratio:K}
  \frac{\length_{\alpha}(X)} {\length_{\widehat{\Phi}(\alpha)}(Y)} \equiv K.
\end{equation}

By  Theorem \ref{thm:linearity:rigidity} and Lemma \ref{lem:patterson}, we see that $\widehat{\Phi}$ is compatible with Dehn twists. Namely,
\begin{equation}\label{eq:compatibility}
  \widehat{\Phi}(D^n_\alpha\beta)=
  D^n_{\widehat{\Phi}(\alpha)}\widehat{\Phi}(\beta)
\end{equation} holds for  any integer $n$ and  any simple closed curves $\alpha$ and $\beta$, where  $D^n_\alpha\beta$ is the $n$-th Dehn twist of $\beta$ along $\alpha$. Considering the length functions, we have
  \begin{equation}\label{eq:ratio:DTY}
  \lim_{n\to\infty}\frac{\length_{D^n_{\widehat{\Phi}(\alpha)}
  \widehat{\Phi}(\beta)}(Y)} {|n|\length_{\widehat{\Phi}(\alpha)}(Y) i(\widehat{\Phi}(\alpha), \widehat{\Phi}(\beta)) } =1
\end{equation}
and
 \begin{equation}\label{eq:ratio:DTX}
  \lim_{n\to\infty}\frac{\length_{D^n_{\alpha}
  \beta}(X)} {|n|\length_{\alpha}(X) i(\alpha,\beta)} =1,
\end{equation}
where $i(\cdot,\cdot)$ represents the geometric intersection number.

Suppose to the contrary that $\widehat{\Phi}$ sends some non-separating curve $\alpha$ to a separating curve $\widehat{\Phi}(\alpha)$. Let $\gamma$ be a separating curve on $S_{1,2}$ with $i(\alpha,\gamma)=0$. Consider the component   of $S_{1,2}-\gamma$ which contains $\alpha$, and the component  of $S_{g',n'}-\widehat{\Phi}(\gamma)$ which contains $\widehat{\Phi}(\alpha)$. By Remark \ref{rmk:torus:sphere}, we see that $\widehat{\Phi}$ induces an isomorphism $\C(S_{1,1})\to\C(S_{0,4})$ which is still denoted by $\widehat{\Phi}$ for simplicity.
Let $\beta$ be simple closed curve on $S_{g',n'}-\widehat{\Phi}(\gamma)$ with $i(\alpha,\beta)=1$. Then $\widehat{\Phi}(\beta)$ is a simple closed curve on  $S_{g',n'}-\widehat{\Phi}(\gamma)$ with
$i(\widehat{\Phi}(\alpha), \widehat{\Phi}(\beta)) =2$. Combined with Equations (\ref{eq:compatibility}), (\ref{eq:ratio:DTY}) and (\ref{eq:ratio:DTX}), this implies  that
$$ \lim_{n\to\infty}
\frac{\length_{D^n_\alpha\beta}(X)}
{\length_{\widehat{\Phi}(D^n_\alpha\beta)} (Y)} = \frac{1}{2} \frac{\length_\alpha(X)}{\length_{\widehat{\Phi}(\alpha)}(Y)}, $$
which contradicts Equation (\ref{eq:ratio:K}).
This completes the proof.
\end{proof}

 \section{Proof of Theorems }
 In this final section, we shall prove Theorem \ref{thm:maintheorem1}, \ref{thm:topology}, \ref{thm:localrigidity}, \ref{thm:isometry:group},  \ref{thm:topology:global}, which are  stated in the introduction.
 \begin{proof}[Proof of Theorem \ref{thm:maintheorem1}]
 We divide the proof into two cases depending on whether $(g,n)$ belongs to $\{(0,4),(1,1)\}$ or not.

   {\textbf{Case 1: $(g,n)\notin\{(0,4),(1,1)\}$.}}

  By duality, $\Phi^*$ induces an $\R$-linear surjective isometry
  \begin{equation*}
    \Phi: (T_Y\TS,\|\bullet\|_{\mathrm{Th}})\to
   (T_X\TS,\|\bullet\|_{\mathrm{Th}}).
  \end{equation*}
  By Proposition \ref{prop:isometry:isomorphism} and Theorem \ref{thm:curvecomplex:rigidity},  we see that  there exists a mapping class $\phi\in\MCG^\pm(S)$ such that   for every simple closed curve $\alpha$
   $$ \Phi(F_X(\alpha)) =F_{Y}(\phi\alpha)),$$
   where $$F_X(\alpha)=\{v\in T^1_X\TS: (d_X\log\ell_\alpha)(v)=1\}$$ and $$F_Y(\phi\alpha)=\{u\in T^1_Y\TS: (d_Y\log\ell_{\phi\alpha})(u)=1\}.$$
   Therefore,
   \begin{equation}\label{eq:phi*}
     \Phi^*(d_X\log\ell_\alpha)=d_Y\log\ell_{\phi\alpha}.
   \end{equation}
   for every simple closed curve $\alpha$.

   To complete the proof, it remains to show that $\phi:X\to Y$ is (isotopic to) an isometry.

    Notice that $\phi$ induces an $\R$-linear surjective isometry:
   \begin{equation*}
     \begin{array}{cccc}
       \widehat{\phi}: &  T^*_Y\TS& \longrightarrow&T^*_{\phi^{-1}Y}\TS \\
        & d_Y\ell_\mu & \longmapsto& d_{\phi^{-1}Y}\ell_{\phi^{-1}\mu}.
     \end{array}
   \end{equation*}
   Consequently,
   \begin{equation*}
     \widehat{\phi}\circ\Phi^*:T_X^*\TS\longrightarrow T_{\phi^{-1}Y}^*\TS
   \end{equation*}
   is also an $\R$-linear surjective isometry. Moreover, by (\ref{eq:phi*}), we see that for every simple closed curve $\alpha$,
   \begin{equation*}
       \widehat{\phi}\circ\Phi^*(d_X\log\ell_\alpha)=d_{\phi^{-1}Y}\log \ell_\alpha.
   \end{equation*}
   Then
   \begin{eqnarray*}
      \widehat{\phi}\circ\Phi^*( d_X\ell_\alpha) &= & \ell_\alpha(X) \cdot ( \widehat{\phi}\circ\Phi^*)( d_X\log \ell_\alpha)\\
      &=&  \ell_\alpha(X) \cdot d_{\phi^{-1}Y}\log \ell_\alpha\\
      &=& \frac{\ell_\alpha(X)}{\ell_{\alpha}(\phi^{-1}Y)}
      d_{\phi^{-1}Y}\ell_{\alpha}
   \end{eqnarray*}
   for every simple closed curve $\alpha$. Since the set of weighted simple closed curves is dense in $\MLS$, it follows that
   $ \widehat{\phi}\circ\Phi^*$ coincides with  $ \Gamma_{X,\phi^{-1}Y}$   on a dense subset of $T^*_X\TS$ (see (\ref{eq:homeo:tangent:XY}) for the definition of  $ \Gamma_{X,\phi^{-1}Y}$).
   By continuity, this implies that
   \begin{equation*}
     \widehat{\phi}\circ\Phi^*=\Gamma_{X,\phi^{-1}Y}.
   \end{equation*}
  In particular, $\Gamma_{X,g^{-1}Y}$ is linear. It then follows from Theorem \ref{thm:linearity:rigidity} that $X=\phi^{-1}Y$.

  \bigskip
   {\textbf{Case 2: $(g,n)\in\{(0,4),(1,1)\}$.}}

   The case that $(g,n)=(1,1)$ is treated in \cite{DLRT2016}.  The proof for the case $(g,n)=(0,4)$ is essentially the same as that of the case $(g,n)=(1,1)$. For completeness, we include a proof for $(g,n)=(0,4)$.

 By Theorem \ref{thm:mostly:facets}, the isometry \begin{equation*}
  \Phi:(T_X\T(S_{0,4}),\|\bullet\|_{\mathrm{Th}})
  \to(T_Y\T(S_{0,4}),\|\bullet\|_{\mathrm{Th}})
\end{equation*}
induces a bijection
$\widehat{\Phi}:\mathscr{C}^0(S_{0,4})\to\mathscr{C}^0(S_{0,4})$
between the sets of simple closed curves such that
$$ \Phi(F_X(\alpha))=F_Y(\widehat{\Phi}(\alpha)). $$

%For convenience, we denote $\widehat{\Phi}(\alpha)$ by $\hat{\alpha}$.
Let $\alpha$ a simple closed curve such that both $\ell_\alpha(X)>L_2$ where $L_2$ is the constant Proposition \ref{prop:correspondence}.    By Proposition \ref{prop:correspondence},  there exist $N>0$,  $n_0\in\mathbb{Z}$,  and a simple closed curve $\beta$ with
$\i(\beta,\alpha)=2$, such that for all $n>N$,
\begin{equation}\label{eq:twist:04}
 \widehat{\Phi}( D^n_\alpha(\beta))=D^{n+n_0}_{\widehat{\Phi}(\alpha)}(\widehat{\Phi}(\beta)). 
\end{equation}
Since $\Phi$ is an isometry, it follows that $|F_X(D^n_\alpha(\beta))|
=|F_Y(D^{n+n_0}_{\widehat{\Phi}(\alpha)}(\gamma))|$.
Combined with  Theorem \ref{thm:facet:length:limit}, this implies that
\begin{equation*}\label{eq:length:correspondence1}
  \ell_\alpha(X)=\ell_{\widehat{\Phi}(\alpha)}(Y).
\end{equation*}

Moreover, there exists $N_2>0$ such that for all $n>N_2$, we have
\begin{equation*}
  \ell_{D^n_\alpha(\beta)}(X)>L_2, ~~
  \ell_{D^{n+n_0}_{\widehat{\Phi}(\alpha)}\gamma}(Y)>L_2.
\end{equation*}
Applying the argument above to $D^n_\alpha\beta$ and $\widehat{\Phi}( D^n_\alpha\beta)=D^{n+n_0}_{\widehat{\Phi}(\alpha)}(\gamma)$, we see that for all $n>N_2$,
\begin{equation}\label{eq:length:correspondence2}
  \ell_{D^n_\alpha(\beta)}(X)=
  \ell_{D^{n+n_0}_{\widehat{\Phi}(\alpha)}\gamma}(Y).
\end{equation}

Let $m>N_2$ be such that $d\ell_\alpha$ and $d\ell_{D^{m}_\alpha\beta}$ are linearly independent in $T^*_X\T(S_{0,4})$.
Since $\i(\alpha,D^{m}_\alpha\beta)=\i(\widehat{\Phi}(\alpha),{D^{m+n_0}_{\widehat{\Phi}(\alpha)}\gamma})=2$.
it follows that there exists a mapping class $f\in\MCGS$ such that $Y=fX$ and
$$ f(\alpha)=\widehat{\Phi}(\alpha), f(D^{m}_\alpha\beta)=\widehat{\Phi}(D^{m}_\alpha\beta)={D^{m+n_0}_{\widehat{\Phi}(\alpha)}\gamma}. $$

Notice that $f$ induces an $\R$-linear isometry
\begin{eqnarray*}
f^*: (T^*_X\T(S_{0,4}),\|\bullet\|_{\mathrm{Th}})&\longrightarrow& (T^*_{fX}\T(S_{0,4}),\|\bullet\|_{\mathrm{Th}})\\
d\ell_\mu&\longmapsto&d \ell_{f\mu}.
\end{eqnarray*}
Consequently, the composition map
\begin{eqnarray*}
(f^*)^{-1}\circ\Phi^*: (T^*_X\T(S_{0,4}),\|\bullet\|_{\mathrm{Th}})&\longrightarrow& (T^*_{X}\T(S_{0,4}),\|\bullet\|_{\mathrm{Th}})
\end{eqnarray*}
is an $\R$-linear isometry with
$$ (f^*)^{-1}\circ\Phi^* (d\ell_\alpha)=d\ell_\alpha,~ (f^*)^{-1}\circ\Phi^* (d\ell_{D^m_\alpha\beta})=d\ell_{D^m_\alpha\beta}. $$
Since $d\ell_\alpha$ and  $d\ell_{D^m_\alpha\beta}$ are  linearly independent, this implies that $(f^*)^{-1}\circ\Phi^*$ is  the identity map. Therefore $\Phi^*=f^*$. This completes the proof.
\end{proof}

\begin{proof}[Proof of Theorem \ref{thm:topology}]
Considering the dimension $\mathrm{Dim}_\R(\TS)=6g-6+2n$, we see that $(g,n)\in\{(0,4),(1,1)\}$ if and only if $(g',n')\in\{(0,4),(1,1)\}$. We divide the proof into two cases depending on whether $(g,n)$ and $(g',n')$ belong to $\{(0,4),(1,1)\}$ or not.

{\textbf{Case 1}: $(g,n)\notin\{(0,4),(1,1)\}$ and $(g',n')\notin\{(0,4),(1,1)\}$.}
  In this case,  Proposition \ref{prop:isometry:isomorphism} implies  that $\Phi^*$ induces an isomorphism $\widehat{\Phi}: \mathscr C(S_{g,n})\to \mathscr C (S_{g',n'})$ which preserves curve type.  It then follows from Theorem \ref{thm:curvecomplex:rigidity} that either \begin{itemize}
     \item $S_{g,n}$ and $S_{g',n'}$ are homeomorphic, or
     \item $\{S_{g,n},S_{g',n'}\}$ is one of $\{\{S_{0,5},S_{1,2}\},\{S_{0,6},S_{2,0}\}\}$.
   \end{itemize}
   On the other hand,  $S_{1,1}$ (resp. $S_{2,0}$) contains non-separating simple closed curves, while all simple closed curves on $S_{0,5}$ (resp. $S_{0,6}$) are separating. This implies that $\{S_{g,n},S_{g',n'}\}$ is neither $\{S_{0,5},S_{1,2}\}$ nor $\{S_{0,6},S_{2,0}\}$. Therefore, $S_{g,n}$ and $S_{g',n'}$ are homeomorphic.

 {\textbf{Case 2}: $(g,n)\in\{(0,4),(1,1)\}$ and $(g',n')\in\{(0,4),(1,1)\}$.}
  Suppose to the contrary that $S_{g,n}$ and $S_{g',n'}$ are not homeomorphic. Then $\{S_{g,n},S_{g',n'}\}=\{(0,4),(1,1)\}$. Without loss of generality we may assume that $S_{g,n}=S_{0,4}$ and $S_{g',n'}=S_{1,1}$.
  By Theorem \ref{thm:mostly:facets}, the isometry \begin{equation*}
  \Phi:(T_X\T(S_{0,4}),\|\bullet\|_{\mathrm{Th}})
  \to(T_Y\T(S_{1,1}),\|\bullet\|_{\mathrm{Th}})
\end{equation*}
induces a bijection
$\widehat{\Phi}:\mathscr{C}^0(S_{0,4})\to\mathscr{C}^0(S_{1,1})$
between the sets of simple closed curves such that
$$ \Phi(F_X(\alpha))=F_Y(\widehat{\Phi}(\alpha)). $$

%For convenience, we denote $\widehat{\Phi}(\alpha)$ by $\hat{\alpha}$.
Let $\alpha$ a simple closed curve such that both $\ell_\alpha(X)>L_2$ and $\ell_{\hat{\alpha}}(Y)>L_1$, where $L_1$ and $L_2$ are the constants from Proposition \ref{prop:correspondence1} and Proposition \ref{prop:correspondence}.    By Proposition \ref{prop:correspondence1} and Proposition \ref{prop:correspondence},  there exist $N>0$,  $n_0\in\mathbb{Z}$,  simple closed curves $\beta$ with
$\i(\beta,\alpha)=2$ and $\gamma$ with $\i(\gamma,\widehat{\Phi}(\alpha))=1$, such that for all $n>N$,
\begin{equation*}
 \widehat{\Phi}( D^n_\alpha\beta)=D^{n+n_0}_{\widehat{\Phi}(\alpha)}\gamma.
\end{equation*}
Since $\Phi$ is an isometry, it follows that $|F_X\left(D^n_\alpha\beta\right)=|F_Y\left(D^{n+n_0}_{\widehat{\Phi}(\alpha)}\gamma\right)|$.
Combined with  Theorem \ref{thm:DLRT:facet:limit} and Theorem \ref{thm:facet:length:limit}, this implies that
\begin{equation}\label{eq:length:correspondence1}
  \ell_\alpha(X)=\ell_{\widehat{\Phi}(\alpha)}(Y).
\end{equation}

Moreover, there exists $N_2>0$ such that for all $n>N_2$, we have
\begin{equation*}
  \ell_{D^n_\alpha\beta}(X)>L_2, ~~
  \ell_{D^{n+n_0}_{\widehat{\Phi}\alpha}\gamma}(Y)>L_1.
\end{equation*}
Applying the argument above to $D^n_\alpha\beta$ and $\widehat{\Phi}( D^n_\alpha\beta)=D^{n+n_0}_{\widehat{\Phi}(\alpha)}\gamma$, we see that for all $n>N_2$,
\begin{equation}\label{eq:length:correspondence2}
  \ell_{D^n_\alpha\beta}(X)=
  \ell_{D^{n+n_0}_{\widehat{\Phi}\alpha}\gamma}(Y).
\end{equation}
On the other hand, we have
\begin{equation*}
 \lim_{n\to+\infty} \frac{\ell_{D^n_\alpha\beta}(X)}{n}=2\ell_\alpha(X)
\end{equation*}
and
\begin{equation*}
  \lim_{n\to+\infty}\frac{\ell_{D^{n+n_0}_{\widehat{\Phi}(\alpha)}\gamma}(Y)}{n}=\ell_{\widehat{\Phi}(\alpha)}(Y).
\end{equation*}
These contradict Equations (\ref{eq:length:correspondence1}) and (\ref{eq:length:correspondence2}).  Therefore, $\{S_{g,n},S_{g',n'}\}\neq \{S_{0,4},S_{1,1}\}$.
This completes the proof.
\end{proof}
 
\begin{proof}[Proof of Theorem \ref{thm:localrigidity}]
The proof we present here follows essentially that of  \cite[Theorem 1.5]{DLRT2016}.
Let $F: (U,\dth)\to (\TS,\dth)$ be an isometric embedding. By \cite[Theorem 6.1]{DLRT2016} the Thurston norm is locally Lipschitz (locally $C^{0,1}_{loc}$). It then follows from \cite[Theorem A, Theorem B]{MatveevTroyanov2017} that $F$ is $C^{1,1}_{loc}$ and its differential is norm-preserving. Therefore, for each $X\in U$,
\begin{equation*}
  d_X F: T_X \TS\longrightarrow T_{F(X)}\TS
\end{equation*}
is an $\R$-linear surjective isometry for the Thurston norm. By duality, $d_X F$ induces an $\R$-linear surjective isometry
\begin{equation*}
  d^*_X F: (T^*_{F(X)} \TS,\|\bullet\|_{\mathrm{Th}})\longrightarrow
  (T^*_X\TS,\|\bullet\|_{\mathrm{Th}}).
\end{equation*}
  By Theorem \ref{thm:maintheorem1}, there exists  $\phi_F (X)\in\MCGS$ such that \begin{equation}\label{eq:F:phi}
   F(X)=\phi_F (X)(X).
  \end{equation}

   {Let $j:\MCGS\to\mathrm{Isom}(\TS,\dth)$ be the natural homomorphism. Let $\widehat U\subset U$ be the subset of points with trivial stablizer in $j(\MCGS)$. Then $\widehat U$ is an open and dense subset of $U$. Let $X_0\in \widehat U$. By the proper discontinuity action of $\MCGS$ on $\TS$, we see that there exist a small neighbourhood $U_0\subset U$ of $X_0$ and a small neighbourhood $V_0\subset F(U)$ of $F(X_0)$, such that 
   \begin{equation*}
     \{\phi\in j(\MCGS): \phi(U_0)\cap V_0\neq\emptyset\}=\{j(\phi_F(X_0))\}.
   \end{equation*}
   It then follows from the continuity of $F$ and (\ref{eq:F:phi}) that
   \begin{equation*}
     j(\phi_F(X))=j(\phi_F(X_0))
   \end{equation*}
    for all $X$ sufficiently close to $X_0$. In other words, the map $j\circ\phi_F: \widehat{U}\to \mathrm{Isom}(\TS,\dth)$ is locally constant. Hence, $j\circ \phi_F$ is  constant  on each connected component of $\widehat U$. }
    
    {For each $\psi\in j(\MCGS)$, let $\mathrm{Fix}_\psi\subset \TS$ be the subset of fixed points. Notice that $\mathrm{Fix}_\psi$ is a real submanifold.\footnote{This  can be seen using the Teichm\"uller metric. Since $\TS$ is uniquely geodesic in the Teichm\"uller metric \cite[Section 7.4]{GL2000} and $\MCGS$ acts as isometries with respect to the Teichm\"uller metric, it follows that $\mathrm{Fix}_\psi$ is a totally geodesic submanifold of $\TS$.}  Let $U_1$ be the union of those $\mathrm{Fix}_\psi$ which are of real codimension one. Let $U_{\geq2}$ be the union of those $\mathrm{Fix}_\psi$ which are of real codimension at least two. Combining with the proper continuity action of $\MCGS$ on $\TS$, we see that the complement $U\setminus U_{\geq2}$ is connected in $U$.}
    
    {Let $U'=U_1\setminus U_{\geq2}$.      Let $X'$ be an arbitrary point in $U'$. Taking $W\subset U$ to be a small neighbourhood of $X'$, we may assume that $W\setminus U'$ is contained in $\widehat U$ and  has exactly two components, which we denote by $W_+$ and $W_-$. By the discussion above, we see that $j\circ \phi_F$ is constant on $W_-$ and on $W_+$.  Let $\psi_-$ and $\psi_+$ be the corresponding values. The continuity of $F$ implies that $\psi_+\circ \psi_-^{-1}\in j(\MCGS)$ fixes $W\cap U'$ pointwise, and is therefore the identity or a reflection.  In the latter case, $F$ would map both sides of $W\cap U'$ to the same side of $F(W\cap U')$. This contradicts the fact that $F$ is $C^1$ at $X'$. Therefore, $\psi_+\circ\psi_-^{-1}$ is the identity.  This implies that $j\circ \phi_F$ is locally constant on $\widehat U\cup U'$. Notice that $\widehat U\cup U'=U\setminus U_{\geq2}$, which is both connected and dense by the previous two paragraphs.  Therefore,  $j\circ\phi_F$ extends to a constant function on $U$ with value $j\circ\phi_F(X_0)$. }
 \end{proof}

 \begin{proof}[Proof of Theorem \ref{thm:isometry:group}]
 Let $j:\MCGS\to\mathrm{Isom}(\TS,\dth)$ be the natural homomorphism. By Theorem \ref{thm:localrigidity}, we see that $j$ is surjective. The theorem then follows from the following fact (\cite{Birman1975,Viro1972})
  \begin{equation*}
    \mathrm{Kernel}(j)\simeq\left\{
    \begin{array}{ll}
    \Z_2, &\text{if }  (g,n)\in\{(2,0), (1,1),(1,2)\},\\
    {\Z_2\oplus\Z_2}, &\text{if }  (g,n)=(0,4),\\
      \{\textup{id}\}, &\text{otherwise.}
    \end{array}
    \right.
  \end{equation*}
 \end{proof}

 \begin{proof}[Proof of Theorem \ref{thm:topology:global}]
  This follows  from Theorem \ref{thm:topology} {and \cite[Theorem 6.1]{DLRT2016} }. 
 \end{proof}

%\noindent\textbf{Declarations}
%\begin{enumerate}[-]
 % \item \textbf{Funding} This work is supported by National Natural Science Foundation of China  NSFC 11901241.
 % \item \textbf{Conflicts of interest}  There is no conflict of interest.
 % \item \textbf{Data Availibility Statements}  Data sharing not applicable to this article as no datasets were generated or
 %analysed during the current study.
%\end{enumerate}

%%==============================
 % BibTeX References
%\bibliographystyle{alpha}
% name my BibTeX data bases
%\bibliography{E:/Papers/PHPreference/PHPreferences}
%\nocite{*} %List no cited references
%%===========================

\end{document}